\documentclass[11pt,twosided,letterpaper]{amsart}

\usepackage{amsmath,amssymb}
\usepackage[latin1]{inputenc}
\usepackage{cancel}
\usepackage{amsthm}
\usepackage{dsfont}

\newtheoremstyle{miestilo}{\baselineskip}{3pt}{\itshape}{}{\bfseries}{}{.5em}{}
\newtheoremstyle{miobs}{\baselineskip}{3pt}{}{}{\bfseries}{}{.5em}{}
\theoremstyle{miestilo}
\newtheorem{defn}{Definition}
\newtheorem{teo}[defn]{Theorem}
\newtheorem{cor}[defn]{Corollary}
\newtheorem{lema}[defn]{Lemma}
\newtheorem{prop}[defn]{Proposition}

\theoremstyle{miobs}

\newtheorem*{obs}{Remark}

\newcommand{\N}{\mathbb{N}}

\newcommand{\R}{\mathbb{R}}

\newcommand{\Z}{\mathbb{Z}}

\let\epsilon\varepsilon
\DeclareMathOperator*{\esssup}{ess\,sup}
\newcommand{\vertiii}[1]{{\left\vert\kern-0.25ex\left\vert\kern-0.25ex\left\vert #1 \right\vert\kern-0.25ex\right\vert\kern-0.25ex\right\vert}}

\usepackage{hyperref}
\usepackage[all]{xy}
\calclayout

\begin{document}

\nocite{*}
\title{On an identification of the Lipschitz-free spaces over subsets of $\R^{n}$}
\author{Gonzalo Flores}
\date{March 2017}
\address{Universidad de Chile, Facultad de Ciencias F\'isicas y Matem\'aticas, Departamento de Ingenier\'ia Matem\'atica, Beauchef 851, Santiago, Chile.}
\thanks{The author was supported by the research grants FONDECYT 1130176 and ECOS-CONICYT C14E06.}
\email{gflores@dim.uchile.cl}
\maketitle

\begin{abstract}

In this short note, we develop a method for identifying the spaces $Lip_{0}(U)$ for every nonempty open convex $U$ of $\R^{n}$ and $n\in\N$. Moreover, we show that $\mathcal{F}(U)$ is identified with a quotient of $L^{1}(U;\R^{n})$.
\end{abstract}

\section{Introduction}

Lipschitz-free spaces have been extensively studied in recent literature, but their structure still is not yet completely understood. Nevertheless, there are some results for some specific cases. Here we will give some of those results and the required definitions. These definitions will be given in full generality, but we will focus on the case where the underlying metric space is a subset of a finite-dimensional Banach space.
\\Let $(M,d)$ be a pointed metric space, that is, a metric space with a given distinguished point that will be denoted by $0_{M}$. When there is no confusion, we will just write $0$. We denote the space of real-valued Lipschitz functions over $M$ vanishing at $0$ by $Lip_{0}(M)$, that is
\[ Lip_{0}(M):=\{f\in\R^{M}\,:\, \|f\|_{L}<\infty \,\wedge\, f(0)=0 \}, \]
where $\|f\|_{L}$ denotes the Lipschitz constant of $f$, that is
\[ \|f\|_{L}:=\sup_{\substack{x,y\in M\\x\neq y}} \frac{|f(x)-f(y)|}{d(x,y)}. \]
The function $\|\cdot\|_{L}$ defines a norm over the space $Lip_{0}(M)$ and with this norm is a dual Banach space.
\\We define the evaluation function $\delta_{M}:M\to Lip_{0}(M)^{*}$ as the function such that for every $x\in M$
\[ \langle \delta_{M}(x),f \rangle = f(x). \]
Again, when there is no confusion, we will denote this function simply as $\delta$.
\\The Lipschitz-free space over $M$, denoted by $\mathcal{F}(M)$, is defined as the subspace of $Lip_{0}(M)^{*}$ given by
\[ \mathcal{F}(M):=\overline{\mathrm{span}}\{ \delta(x)\,:\,x\in M\setminus\{0\} \}. \]
It is easy to see that the set $\{ \delta(x)\,:\,x\in M\setminus\{0\} \}$ is linearly independent. Also, it can be shown that this space verifies that $\mathcal{F}(M)^{*}\equiv Lip_{0}(M)$, that is, there exists a linear isometry between these spaces.

\subsection{Previous results}

We now present some results concerning Lipschitz-free spaces, as well as some basics on vector-valued $L^{p}$ spaces and Distribution Theory.

\begin{lema}\label{conmdiag}
Let $M,N$ be two metric spaces, each one with a base point ($0_{M}$ and $0_{N}$, respectively) and $F:M\to N$ a Lipschitz function such that $F(0_{M})=0_{N}$. Then, there existe a unique linear operator $\hat{F}:\mathcal{F}(M)\to\mathcal{F}(N)$ such that $Lip(F)=\|\hat{F}\|$ and $\delta_{N}\circ F=\hat{F}\circ\delta_{M}$, that is, the following diagram conmutes
\[ \xymatrix{
M \ar[r]^{F} \ar[d]^{\delta_{M}} & N \ar[d]^{\delta_{N}} \\
\mathcal{F}(M) \ar[r]^{\hat{F}} & \mathcal{F}(N)
} \]
\end{lema}

Lemma~\ref{conmdiag} shows that the whole Lipschitz structure present in a given metric space is fully absorbed by its Lipschitz-free space, as it assigns a different direction to every point of the metric space. It is worth noticing that if $M$ is a metric subspace of $N$, Lemma~\ref{conmdiag} also says that $\mathcal{F}(M)$ can be seen as a subspace of $\mathcal{F}(N)$. Considering these facts, this gives a way of linearize problems, in the sense that they can be seen as linear problems on the Lipschitz-free spaces. The main problem is that the structure of Lipschitz-free space is not yet fully understood. As an example, it is known that $\mathcal{F}(\R)$ is isometric to $L^{1}(\R)$, but in A. Naor and G. Schechtman proved on \cite{NS} that $\mathcal{F}(\R^{2})$ is not isomorphic to any subspace of $L^{1}$. Also, in the separable case is still open the question of the relation between Lipschitz-equivalence and isomorphisms, that is, if is it true that if two Banach spaces $X,Y$ are Lipschitz-equivalent (there exists a Lipschitz homeomorfism $F:X\to Y$), then they are isomorphic.

In the case when the metric space asociated to the Lipschitz-free space is finite-dimensional, we have the following results

\begin{teo}[Lancien-Perneck\'a, \cite{LP}]
The Lipschitz-free spaces $\mathcal{F}(\ell_{1})$ and $\mathcal{F}(\ell_{1}^{N})$ admit monotone finite-dimensional Schauder decompositions.
\end{teo}

\begin{teo}[H\'ajek-Perneck\'a, \cite{HP}]
Let $X$ be a product of countably many closed (possibly unbounded or degenerate) intervals in $\R$ with endpoints in $\Z\cup\{-\infty,\infty\}$, considered as a metric subspace of $\ell_{1}$ equipped with the inherited metric. Then the Lipschitz-free space $\mathcal{F}(X)$ has a Schauder basis.

In particular, the Lipschitz-free spaces $\mathcal{F}(\ell_{1})$ and $\mathcal{F}(\ell_{1}^{d})$, where $d\in\N$, have a Schauder basis.
\end{teo}

\begin{teo}[Perneck\'a-Smith, \cite{PS}]
Let $N\geq 1$ and consider $\R^{N}$ equipped with some norm $\|\cdot\|$. Let a compact set $M\subseteq\R^{n}$ have the property that given $\xi>0$, there exists a set $\hat{M}\subseteq\R^{N}$ and a Lipschitz map $\Psi:\hat{M}\to M$, such that $M\subseteq \textrm{int}(M)$, $Lip(\Psi)\leq 1+\xi$ and $\|x-\Psi(x)\|\leq\xi$ for all $x\in\hat{M}$. Then the Lipschitz-free space $\mathcal{F}(M)$ has the metric approximation property. In particular, this is true if $M\subseteq\R^{N}$ is compact and convex.
\end{teo}

As we see, there are already results concerning finite-dimensional spaces in some specific cases for general approximation properties. Our goal is to develop a technique for identifying the Lipschitz-free spaces for open convex subsets of finite-dimensional normed spaces. To this end, and in analogy to the case where $n=1$ (as we will see) we need some concepts of measure theory, more precisely the Lebesgue-Bochner spaces. As we can see, for example, in \cite{DU}, the definitions of simple functions, measurable functions and Bochner-integrability are exactly the same, just changing the range of the functions by a Banach space $X$. We will now state the main results that will be needed for the proofs presented here. First, recall Lebesgue Differentiation Theorem

\begin{teo}\label{ldt}
Let $f:\Omega\subseteq\R^{n}\to\R$ be a Lebesgue-integrable function. Then, for almost every $x\in\Omega$ we have that
\[ f(x)=\lim_{\varepsilon\searrow 0} \frac{1}{\lambda^{(n)}(B(0,\varepsilon)\cap\Omega)} \int_{B(0,\varepsilon)\cap\Omega} f d\lambda^{(n)}. \]
\end{teo}

From here on, $X$ is a Banach space and $(\Omega,\Sigma,\mu)$ a measure space.

\begin{teo}
A $\mu$-measurable function $f:\Omega\to X$ is Bochner-integrable if and only if the function $\|f\|:\Omega\to\R$ is integrable.
\end{teo}

\begin{defn}
Let $p\in[1,\infty)$. The Lebesgue-Bochner space $L^{p}(\Omega,\Sigma,\mu;X)$ is the space given by the (equivalence classes of) $\mu$-Bochner-integrable functions $f:\Omega\to X$ such that
\[ \|f\|_{p}:=\left(\int_{\Omega}\|f\|^{p}d\mu\right)^{\frac{1}{p}}<\infty, \]
which becomes a Banach space with the norm $\|\cdot\|_{p}$. Also, we define the Lebesgue-Bochner space $L^{\infty}(\Omega,\Sigma,\mu;X)$ as the space given by the (equivalence classes of) $\mu$-measurable essentially bounded functions, that is such that
\[ \|f\|_{\infty}:=\esssup_{\Omega}\|f\| < \infty. \]
Again, this space becomes a Banach space with the norm $\|\cdot\|_{\infty}$.
\end{defn}

Finally, we recall the definition of the Radon-Nikodym property (from here on, RNP) and some essential results

\begin{defn}
We say that a Banach space $X$ has the RNP if for every $\sigma$-finite measure space $(\Omega,\Sigma,\mu)$ we have that for every $X$-valued absolutely continuous measure $\nu:\Sigma\to X$ of bounded variation, there exists a function $f\in L^{1}(\Omega,\Sigma,\mu;X)$ such that
\[ \nu(A)=\int_{A}f d\mu\,,\,\forall A\in\Sigma. \]
\end{defn}

\begin{teo}[Dunford-Pettis]
Let $X$ be a separable dual space. Then, $X$ has the RNP. In particular, every reflexive Banach space has the RNP.
\end{teo}

\begin{teo}
Let $X$ be a Banach space such that its dual $X^{*}$ has the Radon-Nikodym. Then
\[ L^{p}(\Omega,\Sigma,\mu;X)^{*}\equiv L^{q}(\Omega,\Sigma,\mu;X^{*}), \]
where $p\in[1,\infty)$ and $\frac{1}{p}+\frac{1}{q}=1$.
\end{teo}

It is now clear the importance of these results, since every finite-dimensional Banach space is reflexive for any norm.

\section{Lipschitz-free spaces over finite-dimensional Banach spaces}

Now we give some details on the study of Lipschitz-free spaces over finite-dimensional Banach spaces. As can be seen on \cite{W}, the Lipschitz-free space over $\R$ is already completely identified. Even if it is a known result, we recall the construction, as it will be necessary for the generalization that we will present for the case of upper dimensions.
\\In any case, we will need some results concerning differentiability properties of Lipschitz functions and measure theory. In the following, let $n\geq 1$ be any natural number

\begin{teo}[Rademacher]\label{Rad} Let $U\subseteq\R^{n}$  be an nonempty open set and $f:U\to\R$ be a Lipschitz-function. Then, the set $D_{f}\subseteq U$ where $f$ is Fr\'echet-differentiable has full measure (with respect to the Lebesgue measure over $U$).
\end{teo}

\subsection{The case $n=1$}

In this section, we will construct a bijective linear isometry between the spaces $L^{\infty}(I)$ and $Lip_{0}(I)$ for every nonempty open interval $I\subseteq\R$, in order to conclude that $\mathcal{F}(I)\equiv L^{1}(I)$. We can always assume that $0\in I$.
\\Consider the operator $T:L^{\infty}(I)\to Lip_{0}(I)$ given by
\[ Tg(x):=\int_{0}^{x}g(t) dt,\,\text{for every } x\in I. \]
It is clear that this definition does not depend of the chosen class representant of $g$, and that the integrals are well defined. Also, it is easy to see that this operator is linear. We note now that for every $x,y\in I$ with $x\neq y$ and every $g\in L^{\infty}(I)$ we have that
\[ |Tg(x)-Tg(y)| = \left|\int_{y}^{x}g(t)dt\right| \leq \|g\|_{\infty}|x-y|. \]
Then, $Tg$ is a Lipschitz function with $\|Tg\|_{L}\leq\|g\|_{\infty}$. We see now that $T$ is biyective.
\\Let $g\in L^{\infty}(I)$ be such that $Tg=0$. Then, for every $x,y\in I$ with $x>y$ we will have that
\[ \int_{y}^{x} g(t) dt = Tg(x)-Tg(y) = 0, \]
so we conclude that $g=0$, and $T$ is injective.
\\Now, let $f\in Lip_{0}(I)$ any Lipschitz function. Thanks to the Theorem~\ref{Rad}, we know that $f'$ is defined a.e. on $I$. Moreover, it is easy to show that $f'\in L^{\infty}(I)$ and $\|f'\|_{\infty}\leq\|f\|_{L}$. Finally, we see that for every $x\in I$
\[ Tf'(x) = \int_{0}^{x} f'(t) dt = f(x)-f(0) = f(x). \]
Then, $Tf' = f$ and $T$ is surjective.
\\Finally, we easily note that $T^{-1}f=f'$ and that $T^{-1}$ is continuous, with $\|T^{-1}f\|_{\infty}\leq\|f\|_{L}$. We conclude that $T$ is an isometry between $L^{\infty}(I)$ and $Lip_{0}(I)$.
\\It is known that $L^{1}(I)^{*}\equiv L^{\infty}(I)$. Using this fact, we now show that $T$ is actually $w^{*}$-$w^{*}$ continuous. Let $(g_{n})_{n\in\N}$ be a $w^{*}$-convergent sequence on $L^{\infty}(I)$ and let $g\in L^{\infty}(I)$ be its $w^{*}$-limit. Then, for each $x\in I$ we have that
\[ \langle Tg_{n},\delta(x) \rangle = \int_{0}^{x}g_{n}(t) dt = \langle g_{n},\mathds{1}_{[0,x]} \rangle. \]
As $\mathds{1}_{[0,x]}\in L^{1}(I)$ and $(g_{n})_{n\in\N}$ is $w^{*}$-convergent, we deduce that
\[ \langle Tg_{n},\delta(x) \rangle \rightarrow \langle g,\mathds{1}_{[0,x]} \rangle = \int_{0}^{x}g(t) dt = \langle Tg,\delta(x) \rangle \]
Then, passing through linear combinations and limits, we conclude that $Tg_{n}\overset{*}{\rightharpoonup} Tg$ whenever $g_{n}\overset{*}{\rightharpoonup} g$. We conclude in view of the following well known result.

\begin{prop}
Let $X,Y$ be two Banach spaces. Let $T:Y^{*}\to X^{*}$ be a linear bounded operator. Suppose that $T$ is weak*-weak* continuous. Then, there exists a linear bounded operator $S:X\to Y$ such that $S^{*}=T$. Moreover, if $T$ is a bijective isometry, so is $S$.
\end{prop}

\subsection{The case $n>1$}

Considering the construction for the unidimensional case, our goal is to generalize it for the $n$-dimensional case. The main idea is to use again Rademacher's theorem in order to identify $Lip_{0}(U)$ with some subspace of $L^{\infty}(U;(\R^{n})^{*})$ when $U\subseteq\R^{n}$ is a nonempty open convex set (from here on, $L^{p}(U;X)$ will denote the space $L^{p}(U,\mathcal{L}(U),\lambda^{(n)};X)$, where $\mathcal{L}(U)$ stands for the Lebesgue-measurable sets on $U$, $\lambda^{(n)}$ stands for the $n$-dimensional Lebesgue-measure and $X$ is a Banach space). We may assume that $0\in U$.

\subsubsection{Identification of $Lip_{0}(U)$}

In this section, we show that $Lip_{0}(U)$ is linearly isometric to a subspace of $L^{\infty}(U;(\R^{n})^{*})$. This construction can also be found in the author's master thesis (\cite{tesis}, in spanish).
\\Thanks to Theorem~\ref{Rad}, we know that every Lipschitz function $f:U\to\R$ is Fr\'echet-differentiable almost everywhere (with respect to the Lebesgue measure).
\\Moreover, it's clear that if $f$ is Lipschitz and $x\in U$ is such that $f$ is differentiable at $x$, then $\|\nabla f(x)\|_{*}\leq\|f\|_{L}$ and with this $\nabla f:U\to(\R^{n})^{*}$ (where we assume that $\nabla f(x)=0$ if $f$ is not differentiable at $x$) is an element of $\mathcal{L}^{\infty}(U;(\R^{n})^{*})$.
\\We will first consider the following result from measure theory.

\begin{prop}\label{fmset}
Let $A\subseteq U$ be a set of full measure. Then, for every $x\in U$
\[ A_{x}:=\{y\in U\,:\, x+t(y-x)\in A \text{ a.e. on } [0,1]\} \]
is a set of full measure.
\end{prop}

\begin{proof}
Without loss of generality, we can asume that $x=0$. It suffices to prove that for every $R>0$, $\lambda^{(n)}(B_{R}\cap A_{0})=\lambda^{(n)}(B_{R})$, where $B_{R}=B_{2}(0,R)\cap U$ and $B_{2}(0,R)$ is the ball of $\R^{n}$ with the Euclidean norm. As $A$ is of full measure, we know that $\lambda^{(n)}(B_{R}\cap A)=\lambda^{(n)}(B_{R})$. Then, using spherical coordinates (where $dv$ denotes the surface measure on $\mathbb{S}^{n-1}$) we have that
\[ \int_{\mathbb{S}^{n-1}}\int_{0}^{R(v)} r^{n-1} dr dv = \int_{\mathbb{S}^{n-1}}\int_{0}^{R(v)} \mathds{1}_{A}(rv)) r^{n-1} dr dv, \]
where $R(v)=\sup\{r\in[0,R]\,:\,rv\in U\}$. This yields that necessarily for almost every $v\in\mathbb{S}^{n-1}$
\[ \int_{0}^{R(v)} \mathds{1}_{A}(rv) r^{n-1} dr = \int_{0}^{R(v)} r^{n-1} dr. \]
Let $\Sigma\subseteq\mathbb{S}^{n-1}$ be the set of points where the last equality is true. Considering that, we will have that for every $v\in\Sigma$, $rv\in A$ for almost every $r\in [0,R(v)]$, or equivalently $R(v)v\in A_{0}$. We easily see that $A_{0}$ is star-shaped, so actually we have that $rv\in A_{0}$ for every $r\in[0,R(v)]$, whenever $v\in\Sigma$. Considering this, we deduce that
\[ \lambda^{(n)}(B_{R}\cap A_{0}) = \int_{\mathbb{S}^{n-1}}\int_{0}^{R(v)}\mathds{1}_{A_{0}}(rv) r^{n-1}dr dv \]
\[ = \int_{\Sigma}\int_{0}^{R(v)} r^{n-1} dr dv = \int_{\mathbb{S}^{n-1}}\int_{0}^{R(v)} r^{n-1} dr dv = \lambda^{(n)}(B_{R}), \]
for every $R>0$. Then, $A_{0}$ is of full measure.
\end{proof}

To proceed in analogy to the unidimensional case, consider the following definition

\begin{defn}\label{fg}
For every $x\in U$, we define the operator $T_{x}$ over $L^{\infty}(U;(\R^{n})^{*})$ with values on the quotient of $\R^{U}$ with respect to equality almost everywhere as
\[ (T_{x}g)(y):=\int_{0}^{1}\langle g(x+t(y-x)),y-x \rangle dt \]
for every $y\in U$ such that the integral is well-defined. Otherwise, we set $(T_{x}g)(y)=0$.
\end{defn}

\begin{obs}
In virtue of Proposition~\ref{fmset}, for every $x\in U$ the integrals on the previous definition are well-defined for almost every $y\in U$, because the sets
\[ A=\{ y\in U\,:\, \|g(x+t(y-x))\|_{*}\leq\|g\|_{\infty} \text{ a.e. on } [0,1] \} \]
are of full measure.
\end{obs}

\begin{prop}
For every $x\in U$, $T_{x}$ is well-defined (that is, it does not depend on the chosen class representant) and is linear.
\end{prop}

\begin{proof}
Let $g,h\in L^{\infty}(U;(\R^{n})^{*})$ be such that $g=h$ almost everywhere. We want to prove that $T_{x}g=T_{x}h$ almost everywhere. Thanks to Proposition~\ref{fmset}, we know that the following sets are of full measure
\begin{itemize}
\item $\{y\in U\,:\,\|g(x+t(y-x))\|_{*}\leq\|g\|_{\infty} \text{ a.e. on } [0,1]\}$.
\item $\{y\in U\,:\,\|g(x+t(y-x))\|_{*}\leq\|g\|_{\infty} \text{ a.e. on } [0,1]\}$.
\item $\{y\in U\,:\,g(x+t(y-x))=h(x+t(y-x)) \text{ a.e. on } [0,1]\}$.
\end{itemize}
Then, for every $y\in U$ belonging to the intersection of these three sets (which is again of full measure) we have that
\[ (T_{x}g)(y)=\int_{0}^{1}\langle g(x+t(y-x)),y-x \rangle dt = \int_{0}^{1}\langle h(x+t(y-x)),y-x \rangle dt = (T_{x}h)(y). \]
We deduce that $T_{x}g=T_{x}h$ almost everywhere. The fact that $T_{x}$ is lineal is direct, considering the set of points $y\in U$ such that every evaluation on $y$ is defined by the integral.
\end{proof}

We now recall that the restriction of every Lipschitz function is again a Lipschitz function. In this sense, if $f:U\to\R$ is a Lipschitz function, then we will have that for every $x,y\in U$
\[ \int_{0}^{1}f'(ty;y) dt - \int_{0}^{1}f'(tx;x) dt = f(y)-f(x) = \int_{0}^{1} f'(x+t(y-x);y-x) dt. \]
This last equality imposes a restriction over the elements of $L^{\infty}(U;(\R^{n})^{*})$ to represent the gradient of a Lipschitz function. With this in mind, consider the next definition.

\begin{defn}\label{LipComp}
We say that $g\in L^{\infty}(U;(\R^{n})^{*})$ is \textbf{Lipschitz-compatible} if for almost every $x,y\in U$ (that is, with respect to the product measure over $U\times U$)
\[ (T_{0}g)(y)-(T_{0}g)(x) = (T_{x}g)(y). \]
\end{defn}

\begin{obs}
It is not difficult to see that there are non-trivial functions that are Lipschitz-compatible. As an example, for every $\varphi\in\mathcal{C}^{\infty}_{0}(U)$, $\nabla\varphi$ is Lipschitz-compatible. The property of Lipschitz-compatibility can be understood as the fact that the functions defined by $L_{x}$ and $L_{0}$ are the same (up to some constant) for almost every $x\in U$. Moreover, it is easy to verify that this is actually a property of the class, since $T_{x}$ is linear.
\end{obs}

We would like to determine the image of Lipschitz-compatible functions via the operator $T_{0}$. To study the mentioned image, we need to introduce the next definition

\begin{defn}\label{essLip}
We define the space of \textbf{essentially Lipschitz} functions as
\[ \mathcal{L}ip(U):=\{f\in \R^{U}\,:\, \mathcal{L}(f)<+\infty \}, \]
where the \textbf{essential Lipschitz constant} is defined as
\[ \mathcal{L}(f)=\esssup_{\substack{x,y\in U\\x\neq y}} \frac{|f(y)-f(x)|}{\|y-x\|}, \]
where the essential supremum is taken with respect to the product measure over $U\times U$.
We will also say that $f\in\mathcal{L}ip_{0}(U)$ if $f\in\mathcal{L}ip(U)$ and there exists $K\geq\mathcal{L}(f)$ such that $|f(x)|\leq K\|x\|$ a.e. on $U$.
\end{defn}

\begin{lema}\label{isom}
For every $f\in\mathcal{L}ip_{0}(U)$ we have that
\begin{enumerate}
\item If $h=f$ almost everywhere, then $\mathcal{L}(h)=\mathcal{L}(f)$.
\item There exists a unique $h\in Lip_{0}(U)$ such that $h=f$ almost everywhere.
\end{enumerate}
In particular, $Lip_{0}(U)$ is linearly isometric to the quotient of $\mathcal{L}ip_{0}(U)$ with respect to equality almost everywhere.
\end{lema}

\begin{proof}
First we notice that $\mathcal{L}(\cdot)$ defines a seminorm on $\mathcal{L}ip_{0}(U)$. This is analogous to the fact that
\[ \esssup_{\omega\in\Omega} |g(\omega)| \]
defines a seminorm on $\mathcal{L}^{\infty}(\Omega)$. This directly proves the first part, since $\mathcal{L}$ is zero on almost everywhere null functions. For the second part, we have that there exists a set $F\subseteq U$ of full measure such that
\[ \frac{|f(x)-f(y)|}{\|x-y\|}\leq \mathcal{L}(f) \]
for every $x,y\in F$. Since $F$ is dense in $U$, define $h$ as the unique Lipschitz extension of $f_{F}$. It is clear that $h(0)=0$, because $f\in\mathcal{L}ip_{0}(U)$.
\\For the last part, we notice that from the proof of the second part we can also deduce that $\mathcal{L}(\cdot)$ defines a norm on the mentioned quotient and the operator that maps every $f\in\mathcal{L}ip_{0}(U)$ to its unique Lipschitz representant is linear, since it is defined by density. We deduce that the mentioned operator is in fact a linear isometry thanks to the first part of the lemma.
\end{proof}

Using this space, we can make a relation between Lipschitz-compatibility and Lipschitz functions in terms of the operator $T_{0}$.

\begin{defn}
We define $X$ as the subspace of $L^{\infty}(U;(\R^{n})^{*})$ of Lipschitz-compatible elements.
\end{defn}

\begin{prop}\label{imEssLip}
For every $g\in X$, $T_{0}g\in Lip_{0}(U)$ (in terms of the last isometry).
\end{prop}

\begin{proof}
In virtue of Proposition~\ref{fmset} and the definition of Lipschitz-compatibility, we have that for almost every $x,y\in U$
\[ |(T_{0}g)(y)-(T_{0}g)(x)| = |(T_{x}g)(y)|\leq\|g\|_{\infty}\|y-x\|. \]
Again thanks to Proposition~\ref{fmset}, we also have that for almost every $x\in U$
\[ |(T_{0}g)(x)|\leq \|g\|_{\infty}\|x\|.\]
Then, in virtue of Lemma~\ref{isom}, $T_{0}g\in Lip_{0}(U)$.
\end{proof}

We now prove the main result of this section, which states an identification for the space $Lip_{0}(U)$.

\begin{teo}
The spaces $X$ and $Lip_{0}(U)$ are linearly isometric.
\end{teo}

\begin{proof}
Let $T:X\to Lip_{0}(U)$ be the operator given by $T=T_{0}$. We have already proven that this is a bounded linear operator, with $\|T\|\leq 1$.
\\Consider now the operator $R:Lip_{0}(U)\to X$ given by $Rf=\nabla f$. Using Theorem~\ref{Rad} together with Proposition~\ref{fmset}, we see that this operator is well-defined (that is, $\nabla f$ is Lipschitz-compatible) and it's clearly linear, with $\|R\|\leq 1$ since $\|\nabla f\|_{\infty}\leq\|f\|_{L}$. Moreover, we have that $TR=Id_{Lip_{0}(U)}$ because for every $f\in Lip_{0}(U)$
\[ TRf = T\nabla f = T_{0}(\nabla f) = T_{0}g = f, \]
where the function $g:U\to\R^{n}$ is defined as follows. Let $D_{f}$ be the set of points where $f$ is differentiable
\[ g(x):=\left\{\begin{array}{cll}
\nabla f(x)&,&x\in D_{f}\\
f'(x;x)\cdot\frac{u_{x}}{\|x\|}&,& x\notin D_{f} \wedge x\neq 0 \wedge f'(x;x) \text{ exists}\\
0&,&\text{otherwise}
\end{array}\right. \]
where $u_{x}\in\R^{n}$ is such that $\|u_{x}\|_{*}=1$ and $\langle u_{x},x \rangle = \|x\|$. We see that $g=\nabla f$ almost everywhere and for any $x\in U\setminus\{0\}$, $\langle g(tx),x \rangle = f'(tx;x)$ a.e. on $[0,1]$, and then $T_{0}g=f$.
\\In particular, we have that $T$ is surjective. Suppose now that there exists $x\in U$ such that $Tg(x)>0$. As $f:=Tg\in Lip_{0}(U)$, we will have that there exists $\delta>0$ such that for almost every $y\in B(x,\delta)$
\[ 0<f(y)=\int_{0}^{1}\langle g(ty),y \rangle dt, \]
that is, for almost every $y\in B(x,\delta)$ there exists a positive measure subset of $[0,1]$ such that $g(ty)\neq 0$ on that subset. This implies that there exists a positive measure subset of $U$ such that $g\neq 0$ on that subset. Then, $T$ is injective. We deduce that $T$ is bijective with $T^{-1}=R$. Then $T$ is a linear isometry between $X$ and $Lip_{0}(U)$.
\end{proof}

\subsubsection{Identification of $\mathcal{F}(U)$}

As in the unidimensional case, we would like to use the operator $T$ from the previous section to obtain a linear isometry between $\mathcal{F}(U)$ and a primal space for $X$. For the unidimensional case we used the known fact that $L^{1}(I)$ is the unique (up to isometry) primal of $L^{\infty}(I)$ for any open interval $I\subseteq\R$. In our case we only know that $X$ is a dual space. To find a primal space for $X$, we first recall a classical result on Banach spaces theory

\begin{prop}\label{dualQuo}
Let $Z$ be a Banach space and $V$ a closed subspace of $Z$. Then the dual space $(Z/V)^{*}$ is linearly isometric to the annihilator of $V$
\[ V^{\perp}=\{z^{*} \in Z^{*}\,:\, \langle z^{*},v \rangle, \forall v\in V \}. \]
\end{prop}

Considering this, we will compute the space $V:=X^{\perp}\cap L^{1}(U;\R^{n})$. In order to do this, we show a relation between $\mathcal{C}_{0}^{\infty}(U)$ and $X$. To this end, we will need the next lemma

\begin{lema}\label{dens}
Let $g\in X$ be any function and $(\varphi_{k})_{k\in\N}$ a sequence on $\mathcal{C}^{\infty}_{0}(U)$ with $\varphi_{k}(0)=0$ for every $k\in\N$. Then, the following propositions are equivalent
\begin{enumerate}
\item $(\exists f\in Lip_{0}(U),K>0)\, \varphi_{k}\overset{p.w.}{\longrightarrow} f \wedge \|\varphi_{k}\|_{L}\leq K \wedge g=\nabla f$.
\item $\nabla\varphi_{k} \overset{*}{\rightharpoonup} g \text{ on } L^{\infty}(U;(\R^{n})^{*})$.
\end{enumerate}
\end{lema}

\begin{proof}
\textbf{[1 $\Rightarrow$ 2]} Suppose that $\|\varphi_{k}\|_{L}\leq K$ and that $\varphi_{k}$ converges pointwise to $f$, with $f$ such that $\nabla f=g$. It suffices to show that for every open bounded $n$-dimensional cube $H$ we have that
\[ \int_{H}\frac{\partial\varphi_{k}}{\partial x_{1}} d\lambda^{(n)} \longrightarrow \int_{H}\frac{\partial f}{\partial x_{1}} d\lambda^{(n)}. \]
Let $I=(a,b)$ be an interval such that $H=I\times H'$, where $H'$ is an open bounded $(n-1)$-dimensional cube. Then
\[ \int_{H}\frac{\partial\varphi_{k}}{\partial x_{1}} d\lambda^{(n)} = \int_{H'}\int_{a}^{b} \frac{\partial \varphi_{k}}{\partial x_{1}}(t,y) dt dy = \int_{H'} \left(\varphi_{k}(b,y)-\varphi_{k}(a,y)\right) dy. \]
We note that for every $y\in H'$, $\varphi_{k}(b,y)-\varphi_{k}(a,y)\longrightarrow f(b,y)-f(a,y)$. Moreover, we have that
\[ |\varphi_{k}(b,y)-\varphi_{k}(a,y)| \leq \|\varphi_{k}\|_{L}(b-a)\|e_{1}\|\leq K(b-a)\|e_{1}\|. \]
As $H'$ has finite measure, we have in virtue of the Dominated Convergence Theorem that
\[ \int_{H}\frac{\partial\varphi_{k}}{\partial x_{1}} d\lambda^{(n)} \longrightarrow \int_{H'} \left(f(b,y)-f(a,y)\right) dy, \]
but we also have that
\[ \int_{H'} (f(b,y)-f(a,y)) dy = \int_{H'}\int_{a}^{b} \frac{\partial f}{\partial x_{1}}(t,y) dt dy = \int_{H} \frac{\partial f}{\partial x_{1}} d\lambda^{(n)}, \]
which proves the direct implication.\\
\textbf{[2 $\Rightarrow$ 1]} Suppose that $\nabla \varphi_{k} \overset{*}{\rightharpoonup} g$ on $L^{\infty}(U;(\R^{n})^{*})$. We want to prove that $\varphi_{k}$ converges pointwise to some Lipschitz function. Let $x\in U\setminus\{0\}$ be any point and consider $V=\mathrm{span}\{x\}^{\perp}$. For any $\varepsilon>0$, we denote by $B_{V}(\varepsilon)$ the restriction to $V$ of the ball of radius $\varepsilon$ centered at $0$. Let $\lambda'$ be the $(n-1)$-dimensional Lebesgue measure over $V$. Then, in virtue of Theorem~\ref{ldt}, we have that 
\[ \varphi_{k}(x) = \lim_{\varepsilon\to 0} \frac{1}{\lambda'(B_{V}(\varepsilon))} \int_{B_{V}(\varepsilon)} \left(\varphi_{k}(x+y)-\varphi_{k}(y)\right) dy \]
\[ = \lim_{\varepsilon\to 0} \frac{1}{\lambda'(B_{V}(\varepsilon))} \int_{B_{V}(\varepsilon)}\int_{0}^{1} \langle \nabla\varphi_{k}(y+tx),x \rangle dt dy \]
\[ = \lim_{\varepsilon\to 0} \int_{H_{\varepsilon}} \left\langle \nabla\varphi_{k}, \frac{x}{\|x\|_{2}\lambda'(B_{V}(\varepsilon))} \right\rangle d\lambda = \lim_{\varepsilon\to 0} \langle \nabla\varphi_{k}, f_{\varepsilon} \rangle, \]
where 
\[ H_{\varepsilon}=\{z\in U\,:\,(\exists y\in B_{V}(\varepsilon),t\in[0,1])\,z=y+tx \} \]
and
\[ f_{\varepsilon} = \frac{x}{\|x\|_{2}\lambda'(B_{V}(\varepsilon))}\mathds{1}_{H_\varepsilon}. \]
We easily note that $f_{\varepsilon}\in L^{1}(U;\R^{n})$. In fact, it suffices to notice that given the definition of $H_{\varepsilon}$, $\lambda(H_{\varepsilon})=\|x\|_{2}\lambda'(B_{V}(\varepsilon))$. Then, for any $\eta>0$ and $k,j\in\N$ we have that
\[ |\varphi_{k}(x)-\varphi_{j}(x)| = \lim_{\varepsilon\to 0} |\langle\nabla\varphi_{k}-\nabla\varphi_{j},f_{\varepsilon}\rangle|. \]
Let $\varepsilon>0$ be such that
\[ \lim_{\varepsilon\to 0} |\langle\nabla\varphi_{k}-\nabla\varphi_{j},f_{\varepsilon}\rangle| \leq |\langle\nabla\varphi_{k}-\nabla\varphi_{j},f_{\varepsilon}\rangle|+\eta. \]
Then, there exists $N\in\N$ such that for every $j,k\geq N$
\[ |\varphi_{k}(x)-\varphi_{j}(x)| \leq |\langle\nabla\varphi_{k}-\nabla\varphi_{j},f_{\varepsilon}\rangle|+\eta \leq 2\eta,\]
and we deduce that for every $x\in U$, $(\varphi_{k}(x))_{k\in\N}$ is a Cauchy (hence convergent) sequence. In the following, we denote by $f$ the pointwise limit of this sequence. It is clear that $f(0)=0$. As $\nabla\varphi_{k}$ is $w^{*}$-convergent, $\|\nabla\varphi_{k}\|_\infty=\|\varphi_{k}\|_{L}$ are bounded, say by $K>0$ and then
\[ |f(x)-f(y)| = \lim_{k\to\infty} |\varphi_{k}(x)-\varphi_{k}(y)| \leq \limsup_{k\to\infty} \|\varphi_{k}\|_{L} \|x-y\| \leq K\|x-y\|, \]
so $f\in Lip_{0}(U)$. We now see that $\nabla\varphi_{k}\overset{*}{\rightharpoonup}\nabla f$ because of the direct implication and we conclude that $g=\nabla f$.
\end{proof}

\begin{cor}
The subspace $X':=\{ \nabla\varphi\,:\, \varphi\in\mathcal{C}_{0}^{\infty}(U) \}$ of $L^{\infty}(U;(\R^{n})^{*})$ is $w^{*}$-dense on $X$.
\end{cor}

\begin{proof}
If $g\in X$, take $f=Tg$, a mollifier $(\rho_{k})_{k\in\N}\in \mathcal{C}^{\infty}_{0}(U)$ and define $\varphi_{k}=\rho_{k}*f-\rho_{k}*f(0)$. Then apply Lemma~\ref{dens}. On the other hand, if $\nabla\varphi_{k}\overset{*}{\rightharpoonup} g$ on $L^{\infty}(U;(\R^{n})^{*})$ we again apply Lemma~\ref{dens}, as we can always assume that $\varphi_{k}(0)=0$.
\end{proof}

\begin{prop}
Let $f\in L^{1}(U;\R^{n})$ be any function. Then $f\in X^{\perp}$ if and only if $\nabla\cdot f = 0$ in the sense of distributions.
\end{prop}

\begin{proof}
For every $\varphi\in\mathcal{C}^{\infty}_{0}(U)$ we have that
\[ \langle\nabla\varphi,f\rangle = \int_{C} \langle\nabla\varphi,f\rangle d\lambda = \sum_{k=1}^{n} \int_{C} \frac{\partial\varphi}{\partial x_{k}}f_{k} d\lambda \]
\[ = -\sum_{k=1}^{n} \int_{C} \varphi\frac{\partial f_{k}}{\partial x_{k}} d\lambda = \int_{C} \varphi(\nabla\cdot f) d\lambda. \]
As the derivatives of $\mathcal{C}_{0}^{\infty}(U)$ functions are $w^{*}$-dense on $X$, we conclude the desired equivalence.
\end{proof}
Using Proposition~\ref{dualQuo}, now we know that $(L^{1}(U;\R^{n})/V)^{*}$ is linearly isometric to $V^{\perp}=\overline{X}^{w^{*}}=X$. In other words, $L^{1}(U;\R^{n})/V$ is a primal for $X$, where
\[V=\{f\in L^{1}(U;\R^{n})\,:\,\nabla\cdot f=0 \text{ in the sense of distributions}\}.\]
To conclude, we show that the operator $T$ defined in the previous section is continuous when we equip $Lip_{0}(U)$ and $X$ with their $w^{*}$-topologies, both spaces seen as the dual spaces of $\mathcal{F}(U)$ and $L^{1}(U;\R^{n})$, respectively.

\begin{prop}\label{wwcont}
The linear isometry $T:X\to Lip_{0}(U)$ is $w^{*}$-$w^{*}$ continuous.
\end{prop}

\begin{proof}
This is equivalent to prove that $R=T^{-1}$ is $w^{*}$-$w^{*}$ continuous. Let $(f_{k})_{k\in\N}$ be a sequence in $Lip_{0}(U)$ such that $f_{k}\overset{*}{\rightharpoonup}f$. We want to prove that
$\nabla f_{k}\overset{*}{\rightharpoonup}\nabla f$, considering $X$ as the dual of $L^{1}(U;\R^{n})/V$. That is, for every $[h]\in L^{1}(U;\R^{n})/V$
\[ \langle\nabla f_{k},[h] \rangle \rightarrow \langle \nabla f,[h] \rangle. \]
Noting that this is equivalent to $\langle\nabla f_{k},h \rangle \rightarrow \langle\nabla f,h \rangle$ for every $h\in L^{1}(U;\R^{n})$, we can use Lemma~\ref{dens}. Then, it suffices to prove pointwise convergence and boundedness of $f_{k}$. These two are trivial from the fact that $f_{k}\overset{*}{\rightharpoonup} f$, since $Lip_{0}(U)\equiv \mathcal{F}(U)^{*}$.
\end{proof}

With the last proposition, we conclude that $L^{1}(U;\R^{n})/V \equiv \mathcal{F}(U)$. We summarize the result in the following last theorem.

\begin{teo}
Let $U$ be an open convex subset of $\R^{n}$. Then, $\mathcal{F}(U)$ is linearly isometric to $L^{1}(U;\R^{n})/V$, where $V$ is the subspace of $L^{1}(U;\R^{n})$ given by the functions with null divergence in the sense of distributions. Moreover, if $S$ is the preadjoint of $T$ and $\Psi:U\to\mathcal{C}^{\infty}_{0}(U)^{*}$ is such that
\[ \langle \Psi(x),\varphi \rangle = \varphi(x)\,,\,\forall x\in U,\varphi\in\mathcal{C}_{0}^{\infty}(U), \]
then $S\delta(x)=[f]$ if and only if $\nabla\cdot f = \Psi(0)-\Psi(x)$.
\end{teo}

\begin{proof}
The first part is direct from Proposition~\ref{wwcont}. For the final part, let $f\in L^{1}(U;\R^{n})$ and $x\in U$. Then
\[ I\delta(x) = [f] \Leftrightarrow (\forall \varphi\in\mathcal{C}^{\infty}_{0}(U))\, \langle \nabla\varphi,I\delta(x) \rangle = \langle\nabla\varphi,f\rangle. \]
Then, let $\varphi\in\mathcal{C}^{\infty}_{0}(U)$ be any function. We see that
\[ \langle\nabla\varphi,I\delta(x)\rangle = \langle TT^{-1}(\varphi-\varphi(0)),\delta(x)\rangle = \varphi(x)-\varphi(0) = \langle\Psi(x)-\Psi(0),\varphi\rangle. \]
On the other hand
\[ \langle \nabla\varphi,f \rangle = \sum_{k=1}^{n} \int_{U} \frac{\partial\varphi}{\partial x_{k}}f_{k} d\lambda^{(n)} = -\int_{U}\varphi\sum_{k=1}^{n}\frac{\partial f_{k}}{\partial x_{k}} d\lambda^{(n)} = -\langle \nabla\cdot f,\varphi \rangle. \]
Then, we have that
\[I\delta(x)=[f] \Leftrightarrow \nabla\cdot f=\Psi(0)-\Psi(x). \]
\end{proof}

\section*{Acknowledgements}

The author would like to thank Aris Daniilidis and Anton\'in Proch\'azka for the discutions around the methods shown here for the identification of $Lip_{0}(U)$. Also mention the work done by Marek C\'uth, Ond\v{r}ej F. K. Kalenda and Petr Kaplick\'y in \cite{2016arXiv161003966C} where they obtained the same results independently using another technique to identify the space $Lip_{0}(U)$.

\bibliographystyle{plain}
\bibliography{articulo}

\end{document}